\documentclass[12pt]{article}

\usepackage{amsmath}
\usepackage{amsfonts}
\usepackage{amssymb}
\usepackage{wasysym}
\usepackage{amsthm}
\usepackage[margin=1.0in]{geometry}

\title{Linear automorphisms of vertex operator algebras associated with formal 
changes of variable and Bernoulli-type numbers}
\author{Robert McRae}
\date{}

    \theoremstyle{definition}\newtheorem{rema}{Remark}[section]
    \theoremstyle{plain}\newtheorem{propo}[rema]{Proposition}
    \newtheorem{theo}[rema]{Theorem}
    
    \newtheorem{lemma}[rema]{Lemma}
    \newtheorem{corol}[rema]{Corollary}
    \theoremstyle{definition}

\begin{document}
\bibliographystyle{alpha}
\maketitle

\newcommand{\nordcirc}{\mbox{\tiny ${\circ\atop\circ}$}}
\numberwithin{equation}{section}

\begin{abstract}
 \noindent We study a certain linear automorphism of a vertex operator algebra 
induced by the formal change of variable $f(x)=e^x-1$ and describe 
examples showing how this relates the theory of vertex operator algebras to 
Bernoulli numbers, Bernoulli polynomial values, and related numbers. We give a 
recursion for such numbers derived using vertex operator relations, and study 
the Jacobi identity for modified vertex operators that was introduced in work 
of 
Lepowsky.
\end{abstract}

\setcounter{equation}{0}
\section{Introduction}

In \cite{Z}, Zhu introduced a natural linear automorphism of a vertex operator 
algebra associated with the formal change of variable
\begin{equation}
 e^x-1=x+\sum_{x\geq 2} \dfrac{1}{n!} x^n
\end{equation}
in the course of proving modularity of graded traces for vertex operator 
algebra modules. In fact, Huang showed more generally in the monograph \cite{H} 
that 
any formal change of variable
\begin{equation}
 f(x)=a_1 x+\sum_{n\geq 2} a_n x^n
\end{equation}
where $a_1\in\mathbb{C}^\times$ and $a_n\in\mathbb{C}$ induces a natural linear 
automorphism $\phi_f$ of a vertex operator algebra $V$, and he removed a 
technical 
assumption in \cite{Z} that $V$ be generated as a Virasoro algebra module by 
lowest weight vectors for the Virasoro algebra. Specifically, given 
$f(x)$, there are unique $A_j\in\mathbb{C}$ for $j\geq 1$ such that
\begin{equation}\label{sum}
 \mathrm{exp}\left(\sum_{j\geq 1} A_j 
x^{j+1}\dfrac{d}{dx}\right)a_1^{x\frac{d}{dx}}\cdot x=f(x),
\end{equation}
where for a derivation $\mathcal{D}$ of $\mathbb{C}[[x]]$ and 
$a\in\mathbb{C}^\times$,
\begin{equation}
 a^\mathcal{D}\cdot f(x)=a^\lambda f(x)
\end{equation}
if $f(x)$ is an eigenvector for $\mathcal{D}$ with eigenvalue $\lambda$. Recall 
\cite{FLM} that a vertex operator algebra $V$ admits an action of the Virasoro 
algebra, the essentially
unique central extension of the Lie algebra of derivations of 
$\mathbb{C}[x,x^{-1}]$. Then the linear automorphism $\varphi_f$ of $V$ is 
given by
\begin{equation}\label{linautsum}
 \varphi_f =\mathrm{exp}\left(-\sum_{j\geq 1} A_j L(j)\right) a_1^{-L(0)},
\end{equation}
where $L(j)$ is the Virasoro algebra element corresponding to the derivation 
$-x^{j+1}\frac{d}{dx}$ of $\mathbb{C}[x,x^{-1}]$; $\varphi_f$ is well defined 
because for any $v\in V$, $L(j)v=0$ for $j$ sufficiently large.

The linear automorphism $\varphi_f$ induces another vertex operator algebra 
structure on $V$: For $u,v\in V$,
\begin{equation}
 Y_f(u,x)v=\varphi_f (Y(\varphi_f^{-1}u,x)\varphi_f^{-1}v)
\end{equation}
 where $Y$ denotes the original vertex operator on $V$: for $u,v\in V$,
 \begin{equation}
  Y(u,x)v=\sum_{n\in\mathbb{Z}} u_n v\,x^{-n-1}
 \end{equation}
 with $u_n\in\mathrm{End}\,V$. The 
new vertex operator $Y_f$ defines a vertex operator algebra structure on $V$ 
that is isomorphic to the original one, with $\varphi_f$ as the isomorphism.
In the case that $f(x)=\frac{1}{a}(e^{ax}-1)$ for $a\in\mathbb{C}^\times$, 
$Y_f$ 
has a particularly simple formula (\cite{Z}, \cite{L}, \cite{H}):
\begin{equation}
 Y_f(u,x)=Y(f'(x)^{L(0)}u,f(x))
\end{equation}
for any $u\in V$. In fact this formula holds for any $f(x)$ if $u$ is a lowest 
weight vector for the Virasoro algebra, that is, $L(0)v=nv$ for some 
$n\in\mathbb{Z}$ and $L(j)u=0$ for $j>0$.

In this paper, we will be particularly concerned with the change of variable 
$f(x)=e^x-1$. One reason is that this function is closely related to the 
Bernoulli numbers and polynomials, as is well known (cf. \cite{IR}). The 
Bernoulli numbers $B_j$ can be defined by the generating function
\begin{equation}
 \dfrac{1}{e^x-1}=\sum_{j\geq 0} \dfrac{B_j}{j!} x^{j-1},
\end{equation}
and more generally, the Bernoulli polynomials $B_n(t)$ are defined by
\begin{equation}
 \dfrac{e^{tx}}{e^x-1}=\sum_{j\geq 0} \dfrac{B_j(t)}{j!}x^{j-1}.
\end{equation}
Then $B_1=-\frac{1}{2}$, and if $j$ is odd and $j>1$, then 
$B_j=0$, because
\begin{equation}
 \dfrac{1}{e^x-1}+\dfrac{1}{2}
\end{equation}
is an odd function. We remark that the Bernoulli numbers $B_j$ are sometimes 
defined alternatively using the generating function
\begin{equation}
 \dfrac{e^x}{e^x-1}=\sum_{j\geq 0} \dfrac{B_j}{j!} x^{j-1},
\end{equation}
but since
\begin{equation}
 \dfrac{e^x}{e^x-1}-\dfrac{1}{e^x-1}=1,
\end{equation}
this changes only $B_1$, which now becomes $+\frac{1}{2}$. We shall use this 
second definition of the Bernoulli numbers in this paper; thus, we take 
$B_j=B_j(1)$ rather than $B_j(0)$.

In Section 2, we present a natural alternative approach to constructing 
$\varphi_f$ for a formal change of variable $f(x)=x+\sum_{n\geq 2} a_n x^n$ by 
defining $b_j\in\mathbb{C}$ for $j\geq 1$ such that
\begin{equation}
 \cdots \mathrm{exp}\left( b_3 x^4 \dfrac{d}{dx}\right) \mathrm{exp}\left( b_2 
x^3 \dfrac{d}{dx}\right)\mathrm{exp}\left( b_1 x^2 \dfrac{d}{dx}\right)\cdot x 
= 
f(x),
\end{equation}
that is, we break the exponential of the infinite sum into an infinite product 
of exponentials. By replacing $-x^{j+1}\frac{d}{dx}$ with $L(j)$, we thus 
obtain 
a linear automorphism $\varphi_f'$ of any vertex operator algebra $V$, which we 
shall show equals $\varphi_f$ using the Campbell-Baker-Hausdorff formula. In 
the 
case that 
$f(x)=e^x-1$, we will also show that $b_1=\frac{1}{2}$ and $b_j=0$ 
for $j>1$ odd, just as with the Bernoulli numbers. Although $b_j\neq B_j$ in 
general, this might suggest an interesting connection between these 
two series of numbers.

In Section 3 we take $f(x)=e^x-1$ and consider $Y_f$ acting on the vertex 
operator 
algebra $V$. If $u\in V$ is homogeneous of conformal weight $m\in\mathbb{Z}$, 
then
\begin{equation}
 Y_f(u,x)=\sum_{n\in\mathbb{Z}} u_n\, e^{mx}(e^x-1)^{-n-1},
\end{equation}
so Bernoulli polynomial values $B_j(m)$ appear when $n=0$. But this expression 
also contains coefficients of general series of the form
\begin{equation}\label{Bernoullitype}
 \dfrac{e^{mx}}{(e^x-1)^{n+1}},
\end{equation}
where $n\in\mathbb{Z}$; we shall call such coefficients Bernoulli-type numbers 
in this 
paper. We will derive a recursion formula for 
Bernoulli-type numbers in the case $m=1$ using certain vertex operator 
relations 
in the 
vertex operator algebra based on the $\mathfrak{sl}_2$ root lattice.

Another motivation for studying the change of variable $f(x)=e^x-1$ is the 
Jacobi identity for the modified vertex operators
\begin{equation}\label{homops}
 X(u,x)=Y(x^{L(0)}u,x),
\end{equation}
which were used extensively in \cite{FLM} for instance, and have the property 
that the 
coefficient of $x^n$ in $X(u,x)$ is an operator of weight $n$ with respect to 
the conformal weight grading of $V$.
In \cite{L2}, Lepowsky showed that these operators are related to the vertex 
operators $Y_f$ for $f(x)=e^x-1$ in that they satisfy the identity
\begin{align}\label{XJacobi}
 x_0^{-1}\delta\left( \frac{x_1-x_2}{x_0}\right) X(u,x_1)X(v,x_2) & -  
x_0^{-1}\delta\left( \frac{-x_2+x_1}{x_0}\right) X(v,x_2)X(u,x_1)\nonumber\\
& =  x_2^{-1}\delta\left( e^y \frac{x_1}{x_2}\right) 
X(Y_f(u,-y)v,x_2),
\end{align}
where $u,v\in V$ and $y=\mathrm{log}\left(1-\frac{x_0}{x_1}\right)$, so that 
$x_0=-x_1(e^y-1)$. Here and elsewhere in this paper $\mathrm{log}(1+X)$ for a 
formal expression $X$ denotes the formal series
\begin{equation}
 \mathrm{log}(1+X)=-\sum_{n\geq 1} \dfrac{(-1)^n}{n} X^n,
\end{equation}
provided this is well defined. See also \cite{M} and \cite{DLM} for generalizations of (\ref{XJacobi}).
To derive consequences of this Jacobi identity for $X$ operators, one needs to 
find the 
coefficients of powers of the formal variables. We will show in Section 4 that 
such coefficients involve formal residues of series of the 
form (\ref{Bernoullitype}); we obtain a formula for these 
residues and derive a formula for the coefficient of any power of $x_0$ in the 
Jacobi identity for modified vertex operators. As a consequence, we recover 
formulas from \cite{FLM}.

\paragraph{Acknowledgments}
I am very grateful to my advisor James Lepowsky for many helpful discussions 
and 
encouragement. I would also like to thank Yi-Zhi Huang, Stephen Miller, and Siddhartha Sahi for 
comments and suggestions regarding this work.

\setcounter{equation}{0}
\section{Linear automorphisms of vertex operator algebras as infinite products}

We take a formal change of variable
\begin{equation*}
 f(x)=x+\sum_{n\geq 2} a_n x^n
\end{equation*}
where $a_n\in\mathbb{C}$; for simplicity we take the coefficient of $x$ to be 
$1$, although this is not completely necessary. We do not assume that the power 
series $f(x)$ converges, since we are 
treating $x$ as a formal variable. We want to write
\begin{equation}\label{product}
 f(x)=\cdots \mathrm{exp}\left( b_3 x^4 \dfrac{d}{dx}\right) \mathrm{exp}\left( 
b_2 
x^3 \dfrac{d}{dx}\right)\mathrm{exp}\left( b_1 x^2 \dfrac{d}{dx}\right)\cdot x ,
\end{equation}
for $b_1, b_2, b_3,\ldots\in\mathbb{C}$ rather than use the form (\ref{sum}). 
Note that such an infinite product is a well-defined operator on, for example, 
$\mathbb{C}[[x]]$ since $x^{j+1}\dfrac{d}{dx}$ for $j\geq 1$ strictly increases 
degrees of monomials. We do not need a factor of the form 
$b_0^{x\, d/dx}$ because we have assumed that $a_1=1$.
\begin{propo}\label{productprop}
 The sequence $\left\lbrace b_j\right\rbrace $ for $j\geq 1$ making 
(\ref{product}) hold exists and is unique.
\end{propo}

\begin{proof}
 Note that the coefficient of $x$ on both the left and right sides of 
(\ref{product}) is $1$.
 The operator $x^{j+1} \dfrac{d}{dx}$ on $\mathbb{C}[[x]]$ for $j\geq 1$ raises 
the degree of any monomial by $j$. Thus the coefficient of $x^{n+1}$, $n\geq 
1$, 
on the right side of (\ref{product}) is determined by the rightmost $n$ factors 
in the product. In fact, this coefficient of is
\begin{equation}
 \sum_{\substack{\mathrm{partitions}\,  p=\\ (1^{i_1},2^{i_2},\ldots 
,n^{i_n})\, 
\mathrm{of}\, n}} 
c_p b_1^{i_1} b_2^{i_2}\cdots b_n^{i_n},
\end{equation}
where $c_p$ is the coefficient of $x^{n+1}$ in the monomial
\begin{equation}
 \frac{1}{i_n !}\left( x^{n+1} \dfrac{d}{dx}\right)^{i_n}\cdots\frac{1}{i_2 
!}\left( x^{3} \dfrac{d}{dx}\right)^{i_2}\frac{1}{i_1 !}\left( x^{2} 
\dfrac{d}{dx}\right)^{i_1} x.
\end{equation}
Since the only partition of $n$ with $n$ as a part is $(n^1)$, we obtain the 
recursive relation 
\begin{equation}\label{recursion}
 b_n=a_{n+1}-\sum_{p\neq (n^1)} c_p b_1^{i_1} b_2^{i_2}\cdots b_n^{i_n}.
\end{equation}
 This recursion 
defines $b_n$ uniquely in terms of $b_1,\ldots, b_{n-1}$.
\end{proof}

Now, since $f(x)$ has the form $f(x)=x+x^2 g(x)$ with $g(x)\in\mathbb{C}[[x]]$, 
$f(x)$ has an inverse $f^{-1}(x)$ of the same form, such that 
$f(f^{-1}(x))=f^{-1}(f(x))=x$. In fact:

\begin{propo}\label{inverseprop}
 We have
\begin{equation}\label{inverseprod}
 f^{-1}(x)=\left( \mathrm{exp}\left( -b_1 x^2 \dfrac{d}{dx}\right) 
\mathrm{exp}\left( -b_2 x^3 \dfrac{d}{dx}\right)\mathrm{exp}\left( -b_3 x^4 
\dfrac{d}{dx}\right)\cdots\right)\cdot x,
\end{equation}
where the $b_j$ are as in (\ref{product}).
\end{propo}

\begin{proof}
The infinite product in (\ref{inverseprod}) is a well-defined operator on 
$\mathbb{C}[[x]]$ because the coefficient of $x^{n+1}$ on the right for any 
$n\geq 0$ is determined by the leftmost $n$ factors.
 Let $L_f$ denote the operator $\cdots\mathrm{exp}\left( b_2 x^3 
\dfrac{d}{dx}\right)\mathrm{exp}\left( b_1 x^2 \dfrac{d}{dx}\right)$ acting on 
$\mathbb{C}[[x]]$, and let $L_f^*$ denote the operator $ \mathrm{exp}\left( 
-b_1 x^2 \dfrac{d}{dx}\right) \mathrm{exp}\left( -b_2 x^3 
\dfrac{d}{dx}\right)\cdots$, which is also well defined, and set 
$f^*(x)=L_f^*(x).$ First observe that 
as operators on $ \mathbb{C}[[x]]$, $L_f L_f^*=L_f^* L_f=1$, $1$ denoting the 
identity. To see this, note that  $L_f L_f^*$ applied to $g(x)\in  
\mathbb{C}[[x]]$, agrees up to degree $n$ with
\begin{equation}\label{inverse}
 \mathrm{exp}\left( b_n x^{n+1} \dfrac{d}{dx}\right)\cdots\mathrm{exp}\left( 
b_1 
x^{2} \dfrac{d}{dx}\right)  \mathrm{exp}\left( -b_1 x^{2} 
\dfrac{d}{dx}\right)\cdots\mathrm{exp}\left( -b_n x^{n+1} \dfrac{d}{dx}\right)
\end{equation} 
applied to $g(x)$. However, (\ref{inverse}) is simply the identity operator; 
since $n$ is arbitrary, we obtain $L_f L_f^*=1$, and similarly we obtain $L_f^* 
L_f=1$. Now, since $x^{n+1} \dfrac{d}{dx}$ is a derivation of 
$\mathbb{C}[[x]]$, 
$\mathrm{exp}\left( b_n x^{n+1} \dfrac{d}{dx}\right) $ is an automorphism, and 
hence $L_f g(x)=g(L_f(x))$ for $g(x)\in  \mathbb{C}[[x]]$. Similarly, $L_f^* 
g(x)=g(L_f^*(x))$. Thus we obtain:
\begin{equation}
 f(f^*(x))=f(L_f^*(x))=L_f^* f(x)=L_f^* L_f (x)=x,
\end{equation}
and similarly $f^*(f(x))=x$. Thus $f^*(x)=f^{-1}(x)$ as desired.
\end{proof}

We can now define a linear automorphism $\widetilde{\varphi}_f$ of any vertex 
operator algebra $V$ as follows:
\begin{equation}
 \widetilde{\varphi}_f = \cdots \mathrm{exp}\left( -b_3 L(3)\right) 
\mathrm{exp}\left( -b_2 
L(2)\right)\mathrm{exp}\left( -b_1 L(1)\right)
\end{equation}
with inverse
\begin{equation}
 \widetilde{\varphi}_f^{-1}=\mathrm{exp}\left( b_1 L(1)\right) 
\mathrm{exp}\left( b_2 
L(2)\right)\mathrm{exp}\left( b_3 L(3)\right)\cdots. 
\end{equation}
These linear maps are well defined because for any $v\in V$, $L(j)v=0$ for $j$ 
sufficiently large. We now prove:
\begin{theo}\label{CBH}
 The linear automorphism $\widetilde{\varphi}_f$ equals the linear automorphism 
$\varphi_f$ defined in (\ref{linautsum}).
\end{theo}
\begin{proof}
 We want to show that for any $v\in V$, 
$\widetilde{\varphi}_f(v)=\varphi_f(v)$. 
Take $N$ such that $L(j)v=0$ for $j>N$ and take formal variables $b_1',\ldots 
b_N'$. Set
 \begin{equation}
  \widetilde{\Phi}_f^{(n)}=\mathrm{exp}\left( -b_N' L(N)\right)\cdots 
\mathrm{exp}\left( -b_2' 
L(2)\right)\mathrm{exp}\left( -b_1' L(1)\right),
 \end{equation}
a well-defined power series. Then by the Campbell-Baker-Hausdorff formula (cf. 
Theorem 3.1 in \cite{BHL}), there are unique 
polynomials $C_j(b_1',\ldots,b_N')\in\mathbb{C}[b_1',\ldots,b_N']$ such that
\begin{equation}\label{CBH2}
 \widetilde{\Phi}_f^{(n)}=\mathrm{exp}\left(-\sum_{j\geq 1} 
C_j(b_1',\ldots,b_N') L(j)\right).
\end{equation}
Since $L(j)\mapsto -x^{j+1}\dfrac{d}{dx}$ defines a representation of the 
positive part of the Virasoro algebra,
\begin{equation}\label{explicitrep}
 \mathrm{exp}\left( b_N' x^{N+1}\dfrac{d}{dx}\right)\cdots \mathrm{exp}\left( 
b_2' 
x^3\dfrac{d}{dx}\right)\mathrm{exp}\left( b_1' x^2\dfrac{d}{dx}\right)\cdot 
x=\mathrm{exp}\left(\sum_{j\geq 1} C_j(b_1',\ldots,b_N') 
x^{j+1}\dfrac{d}{dx}\right)\cdot x.
\end{equation}
Since the coefficient of each power of $x$ in (\ref{explicitrep}) is a 
polynomial in $b_1',\ldots, b_N'$, we can replace $b_j'$ with the complex 
number 
$b_j$ and obtain
\begin{equation}\label{explicitrep2}
 \mathrm{exp}\left( b_N x^{N+1}\dfrac{d}{dx}\right)\cdots \mathrm{exp}\left( 
b_2 
x^3\dfrac{d}{dx}\right)\mathrm{exp}\left( b_1 x^2\dfrac{d}{dx}\right)\cdot 
x=\mathrm{exp}\left(\sum_{j\geq 1} C_j(b_1,\ldots,b_N) 
x^{j+1}\dfrac{d}{dx}\right)\cdot x.
\end{equation}
The left side of (\ref{explicitrep2}) agrees with $f(x)$ up to degree $N$, and 
so the right side does as well. But by the uniqueness of the complex numbers 
$A_j$ defined by (\ref{sum}), this means $A_j=C_j(b_1,\ldots,b_N)$ for $1\leq 
j\leq N$. We thus obtain
\begin{align}
 \widetilde{\varphi}_f(v) &  = \mathrm{exp}\left( -b_N L(N)\right)\cdots 
\mathrm{exp}\left( -b_2 
L(2)\right)\mathrm{exp}\left( -b_1 L(1)\right)\cdot v\nonumber\\
& =\mathrm{exp}\left(-\sum_{j\geq 1} C_j(b_1,\ldots,b_N) L(j)\right)\cdot 
v\nonumber\\
& = \mathrm{exp}\left(-\sum_{j\geq 1} A_j L(j)\right)\cdot v\nonumber\\
& =\varphi_f(v)
\end{align}
since $L(j)v=0$ for $j>N$. Since $v$ was arbitrary, 
$\widetilde{\varphi}_f=\varphi_f$.
\end{proof}
\begin{rema}
 To show the existence of the polynomials $C_j$ in the proof of Theorem 
\ref{CBH}, we are using the easier half of Theorem 3.1 in \cite{BHL}, which 
follows from the Campbell-Baker-Hausdorff formula. The deeper half of Theorem 
3.1 in \cite{BHL} shows that the exponential of a sum such as (\ref{CBH2}) 
can be broken into a product of exponentials.
\end{rema}

Taking $f(x)=\frac{1}{a}(e^{ax}-1)$ where $a\in\mathbb{C}^\times$, we prove the 
following curious result:

\begin{propo}
 Suppose $f(x)=\frac{1}{a} (e^{ax}-1)$, $a\in\mathbb{C}^\times $. Then 
$b_1=\frac{a}{2}$ and $b_j=0$ for $j>1$ odd.
\end{propo}

\begin{proof}
 The fact that $b_1=\frac{a}{2}$ is an immediate consequence of 
(\ref{recursion}). Now, $f^{-1}(x)=\frac{1}{a}\mathrm{log}\left( 1+ax\right)$, 
so by Proposition \ref{inverseprop},
\begin{equation}\label{log}
 \left( \mathrm{exp}\left( -b_2 x^3 \dfrac{d}{dx}\right)\mathrm{exp}\left( -b_3 
x^4 \dfrac{d}{dx}\right)\cdots\right)\cdot x=\mathrm{exp}\left( b_1 x^2 
\dfrac{d}{dx}\right)\cdot \frac{1}{a}\mathrm{log}\left( 1+ax\right)
\end{equation}
 Now, the right side of (\ref{log}) is $\frac{1}{a}\mathrm{log}\left( 
1+a\frac{x}{1-b_1 x}\right)$ by Proposition 8.3.10 in \cite{FLM}. We show that 
this expression contains only odd powers of $x$ by observing that
\begin{align}
 \frac{1}{a}\mathrm{log}\left( 1+a\frac{x}{1-b_1 
x}\right)+\frac{1}{a}\mathrm{log}\left( 1+a\frac{-x}{1+b_1 x}\right)=\nonumber\\
\frac{1}{a}\mathrm{log}\left(\left( 1+\frac{ax}{1-\frac{a}{2} x}\right) \left( 
1-\frac{ax}{1+\frac{a}{2} x}\right) \right) =0.
\end{align}
This follows because
\begin{equation}
 \left( 1+\frac{ax}{1-\frac{a}{2} x}\right) \left( 1-\frac{ax}{1+\frac{a}{2} 
x}\right)=1+\frac{ax\left( 1+\frac{a}{2} x\right)-ax\left( 
1-\frac{a}{2}x\right) 
-a^2 x^2 }{1-\frac{a^2}{4} x^2}=1.
\end{equation}
As in Proposition \ref{productprop}, it is easy to see that we can find unique 
$c_2, c_4,\ldots $ such that
\begin{equation}\label{evenproduct}
  \left( \mathrm{exp}\left( -c_2 x^3 \dfrac{d}{dx}\right)\mathrm{exp}\left( 
-c_4 
x^5 \dfrac{d}{dx}\right)\cdots\right)\cdot x=g(x)
\end{equation}
for any formal change of variable  $g(x)$ that has only odd powers of $x$. 
(Note 
that if $n$ is even, $x^{n+1} 
\dfrac{d}{dx}$ raises powers of $x$ by the even integer $n$, so the left side 
of 
(\ref{evenproduct}) has only odd powers of $x$.)

This shows that we can find $c_2,c_4,\ldots$ such that
\begin{equation}
 \left( \mathrm{exp}\left( -b_1 x^2 \dfrac{d}{dx}\right)\mathrm{exp}\left( -c_2 
x^3 \dfrac{d}{dx}\right)\mathrm{exp}\left( -c_4 x^5 
\dfrac{d}{dx}\right)\cdots\right)\cdot x= \frac{1}{a}\mathrm{log}\left( 
1+ax\right),
\end{equation}
and thus
\begin{equation}
 \cdots \mathrm{exp}\left( c_4 x^5 \dfrac{d}{dx}\right) \mathrm{exp}\left( c_2 
x^3 \dfrac{d}{dx}\right)\mathrm{exp}\left( b_1 x^2 \dfrac{d}{dx}\right)\cdot x 
=\frac{1}{a} (e^{ax}-1).
\end{equation}
By the uniqueness of the coefficients in Proposition \ref{productprop},
$b_j=c_j$ for $j$ even and $b_j=0$ for $j>1$ odd.
\end{proof}

\begin{rema}
 Suppose we take $a=1$ and consider $f(x)=e^x-1$; if we define $b_n'=n!\cdot 
b_n$, the first few non-zero $b_n'$ are given by
 \begin{equation*}
b_1'=\dfrac{1}{2},\,\,\,b_2'=-\dfrac{1}{6},\,\,\,b_4'=-\dfrac{1}{20},\,\,\,
b_6'=\dfrac{5}{84},\,\,\,b_8'=-\dfrac{7}{24},\,\,\,b_{10}'=\dfrac{35}{22},\,\,\,
b_{12}'=-\dfrac{4279}{312},\,\,\,b_{14}'=\dfrac{3003}{16}.
 \end{equation*}
 Based on the experimental evidence, we conjecture that for $n>1$, the sign of 
$b_{2n}$ alternates.
\end{rema}

\setcounter{equation}{0}
\section{A recursion formula for Bernoulli-type numbers}

In this section and the next we will study Bernoulli-type numbers, the 
coefficients of series of the form
\begin{equation}
 \dfrac{e^{mx}}{(e^x-1)^n} =\sum_{j\geq 0} q^{(m,n)}_{-n+j} x^{-n+j},
\end{equation}
where $m,n\in\mathbb{Z}$. Note that for $n=1$, 
\begin{equation}
 q^{(m,1)}_{-n+j}=\frac{B_j(m)}{j!},
\end{equation}
where the $B_j(m)$ are Bernoulli polynomial values; for $m=1$ we get Bernoulli 
numbers $B_j$ (with $B_1=\frac{1}{2}$ rather than $-\frac{1}{2}$). In this 
section, we will take $m=1$ and derive 
a recursion formula for the the values $q^{(1,n)}_{-n+j}$ for any 
$n\in\mathbb{Z}$ and $j\geq 0$ using vertex operator commutation relations. In 
Section 4 we will consider arbitrary $m\in\mathbb{Z}$ but focus on the 
coefficients of $x^{-1}$.

We will take the vertex operator algebra $V$ based on the $\mathfrak{sl}_2$ 
root 
lattice (see \cite{FLM} for the construction of vertex operator algebras from 
lattices). The weight $1$ subspace of $V$ is spanned by the vectors 
$\iota(e_\alpha)$, $\alpha(-1)\mathbf{1}$, and $\iota(e_{-\alpha})$, where 
$\alpha$ represents the positive root of $\mathfrak{sl}_2$. The vertex 
operators 
for these vectors span a copy of the affine Lie algebra 
$\widehat{\mathfrak{sl}_2}$ acting on $V$, so that in particular we have the 
commutation relations:
\begin{equation}\label{6}
 \left[ \alpha(m),\iota(e_{\alpha})_n\right]=2\iota(e_{\alpha})_{m+n} 
\end{equation}
and
\begin{equation}\label{7}
 \left[ \alpha(m),\alpha(n)\right] =m\delta_{m+n,0}
\end{equation}
for any $m,n\in\mathbb{Z}$. We take the formal change of variable $f(x)=e^x-1$ 
with the associated vertex operators
\begin{equation}
 Y_f(u,x)=Y(e^{xL(0)}u,e^x-1),
\end{equation}
which define a vertex operator algebra structure isomorphic to $V$.

We can now derive the following recursive formula for $q^{(1,n)}_{-n+j}$:
\begin{theo}\label{qrecursion}
For any $n\in\mathbb{Z}$, $q^{(1,n)}_{-n}=1$, and for $j> 0$,
\begin{equation}\label{8}
 q^{(1,n)}_{-n+j}=\frac{1}{j}\left( \frac{B_j}{j!}(n-j-1)+\sum_{i=1}^{j-1} i\, 
q^{(1,n+i-j)}_{j-n} q^{(1,-n-i+j+2)}_{n-2}\right) .
\end{equation}
\end{theo}
\begin{proof}
 Since $e^x/(e^x-1)^n=x^{-n}+\ldots $, it is clear that $q^{(1,n)}_{-n}=1$, and 
note also that $q^{(1,n)}_k=0$ for $k<-n$. Since the $Y_f$ vertex operator 
algebra structure on $V$ is isomorphic to the original structure, then denoting 
by $v_{[n]}$ the coefficient of $x^{-n-1}$ in $Y_f(v,x)$, we obtain from 
(\ref{6}) and (\ref{7}) the relations
\begin{equation}\label{9}
 \left[ \alpha[m],\iota(e_{\alpha})_{[n]}\right] =2\iota(e_{\alpha})_{[m+n]}
\end{equation}
and
\begin{equation}\label{10}
 \left[ \alpha[m],\alpha[n]\right] =m\delta_{m+n,0}.
\end{equation}
For $v=\iota(e_{\alpha})$ or $\alpha(-1)\mathbf{1}$, we have
\begin{equation}\label{11}
 Y(v,x)  =  \sum_{j\in\mathbb{Z}} v_j \frac{e^x}{(e^x-1)^{j+1}}
         =  \sum_{k\in\mathbb{Z}} \left( \sum_{j\in\mathbb{Z}} 
q^{(1,j+1)}_{-k-1} v_j\right) x^{-k-1}. 
\end{equation}
Thus we obtain
\begin{align}\label{12}
 \left[ \alpha[m],\iota(e_{\alpha})_{[n]}\right] & = \left[ 
\sum_{k\in\mathbb{Z}} q^{(1,k+1)}_{-m-1} \alpha(k),\sum_{l\in\mathbb{Z}} 
q^{(1,l+1)}_{-n-1} \iota(e_{\alpha})_l\right] \nonumber\\
& =  \sum_{k,l\in\mathbb{Z}} q^{(1,k+1)}_{-m-1} q^{(1,l+1)}_{-n-1} 
(2\iota(e_{\alpha})_{k+l})\nonumber\\
& =  2\iota(e_{\alpha})_{[m+n]}\nonumber\\
& =  2\sum_{j\in\mathbb{Z}} q^{(1,j+1)}_{-m-n-1}\iota(e_{\alpha})_j,
\end{align}
and hence
\begin{equation}\label{13}
 \sum_{j\in\mathbb{Z}} \left( q^{(1,j+1)}_{-m-n-1}-\sum_{k+l=j} 
q^{(1,k+1)}_{-m-1} q^{(1,l+1)}_{-n-1}\right) \iota(e_{\alpha})_j=0.
\end{equation}
Since each $\iota(e_{\alpha})_j$ is an operator of degree $j$ and acts 
non-trivially on $V$, (\ref{13}) can hold only if for each $j\in\mathbb{Z}$,
\begin{equation}\label{14}
 q^{(1,j+1)}_{-m-n-1}=\sum_{k+l=j} q^{(1,k+1)}_{-m-1} 
q^{(1,l+1)}_{-n-1}=\sum_{k\in\mathbb{Z}} q^{(1,k+1)}_{-m-1} 
q^{(1,j-k+1)}_{-n-1}.
\end{equation}
In fact, the sum on the right side of (\ref{14}) is finite since 
$q^{(1,k+1)}_{-m-1}=0$ when $-m-1<-k-1$, that is, $k<m$, and 
$q^{(1,j-k+1)}_{-n-1}=0$ when $-n-1<-j+k-1$, that is, $k>-n+j$. We will use the 
case $j=0$; thus we have
\begin{equation}\label{15}
 \sum_{k=m}^{-n} q^{(1,k+1)}_{-m-1} q^{(1,-k+1)}_{-n-1}=q^{(1,1)}_{-m-n-1}.
\end{equation}
Now for $j\geq 0$, set $m+n=-j$, so that
\begin{equation}\label{16}
 \sum_{k=m}^{m+j} q^{(1,k+1)}_{-m-1} 
q^{(1,-k+1)}_{m+j-1}=q^{(1,1)}_{j-1}=\frac{B_j}{j!}.
\end{equation}

We can also argue similarly using (\ref{7}) and (\ref{10}) that
\begin{equation}\label{17}
 \sum_{k,l\in\mathbb{Z}} q^{(1,k+1)}_{-m-1} q^{(1,l+1)}_{-n-1} 
k\delta_{k+l,0}=m\delta_{m+n,0},
\end{equation}
and so for $m\in\mathbb{Z}$,
\begin{equation}\label{18}
 \sum_{k\in\mathbb{Z}} k\, q^{(1,k+1)}_{-m-1} q^{(1,-k+1)}_{-n-1} 
=m\delta_{m+n,0}
\end{equation}
Again setting $m+n=-j$ for $j\geq 0$, we obtain
\begin{equation}\label{19}
 \sum_{k=m}^{m+j} k\, q^{(1,k+1)}_{-m-1} q^{(1,-k+1)}_{m+j-1} =m\delta_{j,0}.
\end{equation}
Since $q^{(1,m+1)}_{-m-1}=q^{(1,-m-j+1)}_{m+j-1}=1$, (\ref{16}) and (\ref{19}) 
give for $j>0$
\begin{align}\label{20}
 q^{(1,-m+1)}_{m+j-1}+q^{(1,m+j+1)}_{-m-1} & =  
\frac{B_j}{j!}-\sum_{k=m+1}^{m+j-1} q^{(1,k+1)}_{-m-1} 
q^{(1,-k+1)}_{m+j-1}\nonumber\\
m\,q^{(1,-m+1)}_{m+j-1}+(m+j)\,q^{(1,m+j+1)}_{-m-1} & =  -\sum_{k=m+1}^{m+j-1} 
k\,q^{(1,k+1)}_{-m-1} q^{(1,-k+1)}_{m+j-1}
\end{align}
We can solve this system of linear equations and obtain
\begin{align}\label{21}
 \left[\begin{array}{c} q^{(1,-m+1)}_{m+j-1}\\ q^{(1,m+j+1)}_{-m-1}\\ 
\end{array} \right] & =  \frac{1}{j}\left[ \begin{array}{cc} m+j & -1\\ -m & 
1\\ \end{array}\right]\left[ \begin{array}{c} 
\frac{B_j}{j!}-\sum_{k=m+1}^{m+j-1} q^{(1,k+1)}_{-m-1} q^{(1,-k+1)}_{m+j-1}\\ 
-\sum_{k=m+1}^{m+j-1} k\,q^{(1,k+1)}_{-m-1} q^{(1,-k+1)}_{m+j-1}  
\end{array}\right].  
\end{align}
The second component of this equation is
\begin{equation}\label{22}
 q^{(1,m+j+1)}_{-m-1}=-\frac{1}{j}\left( \frac{B_j}{j!} m+\sum_{k=m+1}^{m+j-1} 
(k-m)\,q^{(1,k+1)}_{-m-1} q^{(1,-k+1)}_{m+j-1}\right). 
\end{equation}
First we replace $m$ with $n=m+j+1$ and obtain
\begin{equation}\label{23}
 q^{(1,n)}_{-n+j}=-\frac{1}{j}\left( \frac{B_j}{j!} (n-j-1)+\sum_{k=n-j}^{n-2} 
(k-n+j+1)\,q^{(1,k+1)}_{-n+j} q^{(1,-k+1)}_{n-2}\right).
\end{equation}
Finally, we re-index the sum by replacing $k$ with $i=k-n+j+1$ and obtain
\begin{equation}
 q^{(1,n)}_{-n+j}=-\frac{1}{j}\left( \frac{B_j}{j!} (n-j-1)+\sum_{i=1}^{j-1} 
i\,q^{(1,n+i-j)}_{-n+j} q^{(1,-n-i+j+2)}_{n-2}\right),
\end{equation}
which is the desired result.
\end{proof}
\begin{rema}
 We can use this recursive formula to calculate the first few terms of the 
series expansion of $e^x/(e^x-1)^n$:
\begin{align}\label{24}
 \frac{e^x}{(e^x-1)^n} & =  x^{-n}-\frac{1}{2} (n-2) 
x^{-n+1}+\frac{1}{8}(n-3)\left( n-\frac{4}{3}\right) x^{-n+2}\nonumber\\
&  -\frac{1}{48}(n-1)(n-2)(n-4) x^{-n+3}\nonumber\\
&  +\frac{1}{384}(n-5)\left( n^3-5n^2+\frac{22}{3}n-\frac{16}{5}\right) 
x^{-n+4} \nonumber\\
&  -\frac{1}{3840} (n-1)(n-2)(n-6)\left( n^2-\frac{13}{3} n+\frac{8}{3}\right) 
x^{-n+5}+\ldots
\end{align}
\end{rema}

Motivated by (\ref{24}), we prove:
\begin{propo}
 For any $n\in\mathbb{Z}$ and $j\geq 0$, $q^{(1,n)}_{-n+j}$ is a polynomial in 
$n$ of degree at most $j$. For $j>0$, this polynomial is divisible by $n-j-1$; 
when $j>1$ is odd, it is divisible by $n-1$; and when $j>0$ is odd, it is 
divisible by $n-2$.
\end{propo}
\begin{proof}
 For $j=0$, the result is clear, and for $j=1$, (\ref{8}) gives 
$q^{(1,n)}_{-n+1}=-B_1 (n-2)$. It follows by induction on $j$ that 
$q^{(1,n)}_{-n+j}$ is a polynomial of degree at most $j$ since then for $1\leq 
i\leq j-1$, $q^{(1,n+i-j)}_{-n+j}$ is a polynomial of degree at most $i$ and 
$q^{(1,-n-i+j+2)}_{n-2}$ is a polynomial of degree at most $j-i$. To see why 
the 
polynomial
$q^{(1,n)}_{-n+j}$ is divisible by $n-j-1$ when $j>0$, observe that 
$q^{(1,j+1)}_{-1}$ is the residue of $e^x/(e^x-1)^{j+1}$, but 
$e^x/(e^x-1)^{j+1}$ is the derivative of $-(e^x-1)^{-j}/j$ so 
$q^{(1,j+1)}_{-1}=0$. When $j>1$ is odd, 
$q^{(1,1)}_{-1+j}=B_j/j!=0$, and for $j>0$ odd, $q^{(1,2)}_{-2+j}=0$ because 
$e^x/(e^x-1)^2$ is an even function.
\end{proof}

\begin{rema}
 One might consider trying to use commutation relations other than (\ref{6}) 
and 
(\ref{7}) to 
get information about $q^{(1,n)}_{-n+j}$, but these provide no additional 
information.
\end{rema}

\begin{rema}
 We could use commutation relations among elements of the weight $m$ subspace 
of 
$V$ to obtain information about $q^{(m,n)}_{-n+j}$ for $m>1$.
\end{rema}

\section{Bernoulli-type numbers and the Jacobi identity for modified vertex 
operators}

Suppose $V$ is a vertex operator operator and consider the modified vertex 
operators $X(u,x)$ defined in (\ref{homops}) which satisfy the Jacobi identity 
(\ref{XJacobi}). To obtain consequences of this identity, it is necessary to 
find the coefficients of monomials in the identity. In this section, we will use 
formal residues of the generating functions of Bernoulli-type numbers to 
calculate the coefficient in (\ref{XJacobi}) of $x_0^m$ for any 
$m\in\mathbb{Z}$, thus recovering formulas from \cite{FLM} for these 
coefficients. 

For an arbitrary formal series $g(y)$, we use the residue 
notation 
$\mathrm{Res}_y\,g(y)$ to denote the coefficient of $y^{-1}$ in $g(y)$. We will need the following lemma (see also Lemma 2.4 in \cite{SV}):
\begin{lemma}\label{residue}
  For any $m\in\mathbb{Z}$ and $n>0$,
\begin{equation}
 \mathrm{Res}_y\, \dfrac{e^{my}}{(e^y-1)^n}=  
 \binom{m-1}{n-1}
\end{equation}
\end{lemma}
\begin{proof}
  We use the formal residue change of variable formula
 \begin{equation}\label{changeofvariable}
 \mathrm{Res}_{x}\, g(x)=\mathrm{Res}_y \,g(x(y))x'(y)
\end{equation}
with $x(y)=e^y-1$ to obtain
\begin{align}
 \mathrm{Res}_y\,\dfrac{e^{my}}{(e^y-1)^n} & 
=\mathrm{Res}_y\,\dfrac{e^y(1+(e^y-1))^{m-1}}{(e^y-1)^n}=\mathrm{Res}_x\,\dfrac{
(1+x)^{m-1}}{x^n}\nonumber\\
 &=\mathrm{Res}_x \sum_{k\geq 0}\binom{m-1}{k} 
x^{k-n}=\binom{m-1}{n-1}.
\end{align}
\end{proof}

With this formula for residues of the generating functions of 
Bernoulli-type numbers, we can derive the following formula for the coefficient 
of $x_0^{-n-1}$ in (\ref{XJacobi}):
\begin{propo}
 For any $n\in\mathbb{Z}$ and $u,v\in V$ with $u$ homogeneous,
 \begin{align}\label{gencomm}
  (x_1-x_2)^n X(u,x_1)X(v,x_2) & -(-x_2+x_1)^n X(v,x_2)X(u,x_1)\nonumber\\
  & =  \sum_{j\in\mathbb{Z},\,k\geq n}\binom{\mathrm{wt}\,u-j-1}{k-n} x_1^j 
x_2^{n-j} X(u_k v,x_2).
 \end{align}
\end{propo}
\begin{proof}
 We recall that the Jacobi identity \cite{L2} for modified vertex operators is 
given by
 \begin{align}\label{XJacobi2}
 x_0^{-1}\delta\left( \frac{x_1-x_2}{x_0}\right) X(u,x_1)X(v,x_2) & -  
x_0^{-1}\delta\left( \frac{-x_2+x_1}{x_0}\right) X(v,x_2)X(u,x_1)\nonumber\\
& =  x_2^{-1}\delta\left( e^{y'} \frac{x_1}{x_2}\right) 
X(Y_f(u,-y')v,x_2),
\end{align}
where $y'=\mathrm{log}\left(1-\frac{x_0}{x_1}\right)$, and so
\begin{equation} 
-y'=-\mathrm{log}\left(1-\dfrac{x_0}{x_1}\right)=\mathrm{log}\left(1-\dfrac{x_0}
{x_1}\right)^{-1}=\mathrm{log}\left(\dfrac{x_1}{x_1-x_0}\right).
\end{equation}
The right side of (\ref{XJacobi2}) can be rewritten using $\delta$-function 
substitution properties (see \cite{LL}):
\begin{align}
& 
x_2^{-1}\delta\left(\dfrac{x_1-x_0}{x_2}\right)X\left(Y_f\left(u,\mathrm{log}
\left(\dfrac{x_1}{x_1-x_0}\right)\right)v,x_2\right)=\nonumber\\
& 
x_1^{-1}\delta\left(\dfrac{x_2+x_0}{x_1}\right)X\left(Y_f\left(u,\mathrm{log}
\left(\dfrac{x_2+x_0}{x_2}\right)\right)v,x_2\right)=\nonumber\\
& x_1^{-1}\delta\left(e^y\dfrac{x_2}{x_1}\right)X(Y_f(u,y)v,x_2)
\end{align}
where $y=\mathrm{log}\left(1+\frac{x_0}{x_2}\right)$. Then for any 
$n\in\mathbb{Z}$, the coefficient of $x_0^{-n-1}$ in (\ref{XJacobi2}) is
\begin{align}
 (x_1-x_2)^n X(u,x_1) & X(v,x_2)-(-x_2+x_1)^n X(v,x_2)X(u,x_1)\nonumber\\
 & =\mathrm{Res}_{x_0}\,x_0^n  
x_1^{-1}\delta\left(e^y\dfrac{x_2}{x_1}\right)X(Y_f(u,y)v,x_2)\nonumber\\
 & =\mathrm{Res}_y\,x_2^{n+1} e^y (e^y-1)^n 
x_1^{-1}\delta\left(e^y\dfrac{x_2}{x_1}\right)X(Y_f(u,y)v,x_2)\nonumber\\
 & =\sum_{j,k\in\mathbb{Z}} x_1^{-j-1} x_2^{n+j+1} X(u_k 
v,x_2)\,\mathrm{Res}_y\,\dfrac{e^{(\mathrm{wt}\,u+j+1)y}}{(e^y-1)^{k-n+1}}
\nonumber\\
 & =\sum_{j\in\mathbb{Z},\,k\geq n} x_1^j x_2^{n-j}X(u_k 
v,x_2)\,\mathrm{Res}_y\,\dfrac{e^{(\mathrm{wt}\,u-j)y}}{(e^y-1)^{k-n+1}}
\nonumber\\
 & = \sum_{j\in\mathbb{Z},\,k\geq n}\binom{\mathrm{wt}\,u-j-1}{k-n} x_1^j 
x_2^{n-j} X(u_k v,x_2),
\end{align}
where we have used (\ref{changeofvariable}) with $x_0(y)=x_2(e^y-1)$ and Lemma 
\ref{residue}.
\end{proof}

By taking $n=0$ in (\ref{gencomm}), we obtain a commutator formula for the 
modified vertex operators:
\begin{corol}
 For $u,v\in V$ with $u$ homogeneous,
 \begin{equation}\label{comm}
  [X(u,x_1),X(v,x_2)]=\sum_{j\in\mathbb{Z},\,k\geq 
0}\binom{\mathrm{wt}\,u-j-1}{k}\left(\dfrac{x_1}{x_2}\right)^j X(u_kv,x_2).
 \end{equation}
\end{corol}
\begin{rema}
 By taking coefficients of powers of $x_1$ and $x_2$ in (\ref{gencomm}) and 
(\ref{comm}), we recover formulas (8.8.43) and (8.6.32) of \cite{FLM}, 
respectively.
\end{rema}

\noindent {\small \sc Department of Mathematics, Rutgers University,
110 Frelinghuysen Rd., Piscataway, NJ 08854-8019}
\vspace{1em}

\noindent {\em E-mail address}: rhmcrae@math.rutgers.edu

\end{document}